\documentclass[12pt,a4paper]{article}
\usepackage{amssymb}
\usepackage{amsmath}
\usepackage{amsthm}
\usepackage{mathrsfs}
\usepackage{hyperref}
\usepackage{moresize}
\usepackage{cancel}
\usepackage{bbm}
\usepackage{dsfont}
\usepackage{verbatim}
\usepackage{makeidx}
\usepackage{fullpage}
\newtheorem{theorem}{Theorem}[section]

\newtheorem{lemma}[theorem]{Lemma}
\newtheorem{remark}[theorem]{Remark}
\newtheorem{definition}{Definition}[section]
\newtheorem{example}{Example}[section]
\numberwithin{equation}{section}

\allowdisplaybreaks

\title{The average number of representations of an integer as a sum of two prime powers over multiples of a fixed integer}
\author{Alessandra Migliaccio and Alessandro Zaccagnini}
\date{\today}
\providecommand{\keywords}[1]{\textbf{\textit{Keywords}} #1}

\providecommand{\subjclass}[1]{\textbf{\textit{Classification}} #1}

\begin{document}

\maketitle

\abstract{\noindent We extend a result by Ikeda and Suriajaya (2025) to find the asymptotic behaviour of the average number of representations of an integer $n$, over multiples of a fixed $q\ge 2$, as a sum of two prime $k$-th powers, for $k\ge 2$.}

\newline \noindent

\section{Introduction}\label{sec1}
 We will develop the results of Ikeda and Suriajaya in \cite{ike25}, from the case $k=1$ to $k\ge 2$; namely, we want to estimate the average given by
        \[
        G_{q,k}(N)=\sum_{\substack{n\le N \\ q|n}} \psi_{2,k}(n),
        \]
        where 
        \[
        \psi_{2,k}(n)=\sum_{m_{1}^k + m_{2}^k = n}\Lambda(m_{1})\Lambda(m_{2}). 
        \]
In particular, for $k=1$, we recall the following result of \cite{ike25}.

        \begin{theorem}[Theorem 1.3 of \cite{ike25}]
        \label{th1}
         Assume that GRH holds for Dirichlet L-functions $L(s,\chi)$ associated with characters $\chi\bmod q$. For $2\le q\le N$, we have 
         \[
         G_{q,1}(N)=\frac{G_{1,1}(N)}{\phi(q)} + \mathcal{O}(N\log^3 N).
         \]
        \end{theorem}

        \begin{definition}
        From now on, by GRH we mean that $\mathfrak{R}(\rho)=1/2$, for every non-trivial zero $\rho$ of the Dirichlet L-function $L(s,\chi)$, where $\chi$ mod $q$ is a primitive character or is induced by a primitive character $\chi_{1}$ mod $q_{1}$, s.t. $q_{1}\mid q$.   \end{definition}
        \begin{remark}
        \label{remGRH}
        We will exploit the GRH only in Lemmas \ref{th3} and \ref{lem2}.
        \end{remark}
        \noindent In order to generalize Theorem \ref{th1}, we will mimick the strategy of \cite{ike25}, with some variations. Heuristically, we expect that the order of the main term, in $G_{q,k}(N)$, is given by $G_{1,k}(N)$, which can be roughly estimated as follows:
        \[
        G_{1,k}(N)\asymp  \sum_{n\le N}\left(\sum_{n=m_{1}^k + m_{2}^k} 1\right)\asymp N^{2/k},
        \]
        replacing $\Lambda$ by its average $1$, (for an explicit estimate of $G_{1,2}(N)$, see \cite{lan11}). 
        More precisely, our result is the following.

        \begin{theorem}
        \label{th2}
Let $q\ge 2$ be a fixed integer and $k\ge 2$. Assuming GRH for the modulus $q$, we have that
\[
G_{q,k}(N)=\frac{\Sigma_{k}(q)}{\phi(q)}\cdot G_{1,k}(N)+ \mathcal{O}\left(\frac{N^{1/k}\log^2 N\log q}{\phi(q)}\right),
\]
where $\Sigma_{k}(q):=\sum_{\chi^k =\chi_{0}} \chi(-1)$, as $N\to +\infty$.
\end{theorem}

\begin{remark}
\label{rem-O}
The constant in the $\mathcal{O}$-symbols may depend only on $k$, as we will always make explicit the dependence on $q$. 
\end{remark}

\begin{remark}
\label{rem:dif}
With respect to \cite{ike25}, we notice that a new constant appears, in the main term of $G_{q,k}(N)$. In fact, while $\Sigma_{1}(q)$ is identically equal to $1$, we now have some cases where $\Sigma_{k}(q)=0$; as we will see in Lemma \ref{app}, it is due to the relationship between the value of $\Sigma_{k}(q)$ and the number of solutions of certain equations mod $q$. In other terms, it comes from the fact that $k$-th powers are not uniformly distributed in arithmetic progressions. As a result, we find a main term for every character $\chi$ s.t. $\chi^k =\chi_{0}$, (see Lemma \ref{th3} and \eqref{eq:mathcalI1}), and not a single one.
\end{remark}
\begin{example}
If $q=3$ and $k=2$, we require an integer $n$ s.t. $n\equiv_{3}a^2 +b^2$; however, $\Lambda$-functions, in the definition of $\psi_{2,k}(n)$, force both $a$ and $b$ to be prime powers. As a result, we are looking for integers of the form $n=3^r+3^t$, which are infinitely many but infrequent.
\end{example}
\begin{remark}
\label{remlem}
We adapt Lemmas 2.2 and 2.4 of \cite{ike25} to our case, basing on variations of Lemma 2 of Languasco and Zaccagnini \cite{lan12} and Lemma 2 of Goldston and Vaughan \cite{gol96}.    
\end{remark}

\begin{remark}
\label{rem-j}
In another paper in preparation (see \cite{mig26}), we obtain a further generalization of \cite{ike25} for the representation of an integer as sums of $j\ge 3$ primes.
\end{remark}

\section{Lemmas and sketch of the proof}\label{sec2}
First of all, we define, for $N\ge 4$, a smooth version of $G_{k,q}(N)$ as 
        \[
        F_{q,k}(z)=\sum_{\substack{n \\ q|n}} \psi_{2,k}(n)z^{n},
        \]
        with $z=re(\alpha)$ and $r=e^{-1/N}$, and
        \[
        I_{N} (z) = \sum_{n\le N} z^{n}.
        \]
        For $1\le q\le N$, we have that
        \begin{equation}
        \label{eq:Gqk}
        \begin{split}
        \int_{0}^{1} F_{q,k}(z)I_{N}(z^{-1})\,d\alpha &=\int_{0}^{1} \sum_{\substack{n \\ q|n}} \psi_{2,k}(n)z^{n}\sum_{m\le N} (z^{-1})^{m}\,d\alpha \\
        &=\int_{0}^{1} \sum_{\substack{n \\ q|n}} \psi_{2,k}(n)r^{n}\sum_{m\le N}r^{-m} e(\alpha(n-m))\,d\alpha \\ &=\sum_{\substack{n \\ q|n}} \psi_{2,k}(n) \int_{0}^{1} \sum_{m\le N} r^{n-m} e(\alpha(n-m))\,d\alpha \\
        &=\sum_{\substack{n\le N \\ q|n}} \psi_{2,k}(n)=G_{q,k}(N).
        \end{split}
        \end{equation}
        Furthermore, we notice that, for $q=1$, we have that
        \begin{equation}
        \label{eq:g1k}
        G_{1,k}(N)=\int_{0}^{1} F_{1,k}(z)I_{N}(z^{-1})\,d\alpha=\\ \int_{0}^{1} \Psi_{k}(z)^2 I_{N
        }(z^{-1})\,d\alpha,
        \end{equation}
        where we denote by
        \[
        \Psi_{k}(z)=\sum_{n} \Lambda(n)z^{n^k}.
        \]
        This will provide all the main contributions of our asymptotic formula.
        Then, for a Dirichlet character $\chi\bmod q$, we introduce the function
        \[
        \Psi_{k}(z, \chi)=\sum_{n} \chi(n) \Lambda(n)z^{n^k}.
        \]
        We need an extension of Lemma 1 in Goldston and Suriajaya \cite{gol23}, that lets us approximate $F_{q,k}(z)$ in terms of $\Psi_{k}(z,\chi)$ and  $\Psi_{k}(z,\overline{\chi})$. We can generalize the aforementioned result as follows.
        \begin{lemma}
        \label{th3}
         For $N\ge 4$ and $q\ge 2$, 
        \begin{equation}
        \label{eq:Fqk}
        F_{q,k}(z)=\frac{1}{\phi(q)} \sum_{\chi \bmod q} \chi(-1)\Psi_{k}(z,\chi^k)\Psi_{k}(z,\overline{\chi}^k) + \mathcal{O}((\log N\log 2q)^2).
        \end{equation}
        In particular, for $q=1$, $F_{1,k}(z)=\Psi_{k}(z)^2$.
        \end{lemma}
        \begin{proof}
        \begin{align*}
        F_{q,k}(z)&=\sum_{q|m_{1}^k +m_{2}^k} \Lambda(m_{1})\Lambda(m_{2})z^{m_{1}^k +m_{2}^k} = \sum_{m_{1}^k +m_{2}^k\,\equiv_{q} \,0} \Lambda(m_{1})\Lambda(m_{2})z^{m_{1}^k +m_{2}^k} \\   &=\biggl(\sum_{m_{1}:\,(m_{1},q)=1} + \sum_{m_{1}:\,(m_{1},q)>1}\biggr)\sum_{m_{2}:\,m_{2}^k\,\equiv_{q} \,-m_{1}^k} \Lambda(m_{1})\Lambda(m_{2})z^{m_{1}^k +m_{2}^k} \\ &=:S_{1} + S_{2}.
        \end{align*}
        We recall that, if $(m_{1},q)=1$ and $(m_{1}m_{2},q)=1$, by orthogonality, we have
        \begin{equation}
        \label{eq:ort}
        \begin{split}
        \frac{1}{\phi(q)}\sum_{\chi \bmod q}\chi(-m_{1}^k)\overline{\chi}(m_{2}^k) &= \frac{1}{\phi(q)}\sum_{\chi\bmod q}\chi(-1)\chi(m_{1}^k)\overline{\chi}(m_{2}^k) \\ 
        &= \mathbbm{1}\{{m_{2}^k\equiv_{q}-m_{1}^k \wedge \, (m_{1}m_{2},q)=1}\}.
        \end{split}
        \end{equation}
        Furthermore, we use the fact that $\chi(m^k)=\chi(m)^k$, thus we obtain that
        \begin{align*}
        S_{1}&=\frac{1}{\phi(q)}\sum_{m_{1}, m_{2}}\biggl(\,\sum_{\chi \bmod q}\chi(-1)\chi(m_{1})^k\overline{\chi}(m_{2})^k\biggr)\Lambda(m_{1})\Lambda(m_{2})z^{m_{1}^k + m_{2}^k} \\
        &=\frac{1}{\phi(q)}\sum_{\chi \bmod q}\chi(-1)\sum_{m_{1}}\chi(m_{1})^k\Lambda(m_{1})z^{m_{1}^k}\sum_{m_{2}}\overline{\chi}(m_{2})^k\Lambda(m_{2})z^{m_{2}^k} \\
        &=\frac{1}{\phi(q)}\sum_{\chi \bmod q}\chi(-1)\Psi_{k}(z,\chi^k)\Psi_{k}(z,\overline{\chi}^k).
        \end{align*}
        Instead, for $S_{2}$, we proceed as in $(12)$ of Lemma 1 of \cite{gol23}, therefore we achieve that
        \[
        F_{q,k}(z)=\frac{1}{\phi(q)}\sum_{\chi \bmod q}\chi(-1)\Psi_{k}(z,\chi^k)\Psi_{k}(z,\overline{\chi}^k) + \mathcal{O}((\log N\log 2q)^2).
        \]
        \end{proof}
         \noindent Now, in order to generalize Lemma 2.2 of \cite{ike25}, we need to adapt Lemma 2 of Languasco and Zaccagnini \cite{lan12} to our case.
          \begin{lemma}
        \label{lem1}
        Define the function
        \[
        E_{0}(\chi)=\begin{cases}
         1 & \chi =\chi_{0} \\
        0 & \chi\ne\chi_{0}
        \end{cases}
        \]
        and let $z=e^{-w}$, with $w=N^{-1}-2\pi i\alpha$. Assuming GRH, for every $q\ge 1$, it holds
        \[
        \Psi_{k}(z, \chi)=k^{-1}\Gamma\left(1+\frac{1}{k}\right)^{1/k}\frac{E_{0}(\chi)}{w^{1/k}} - k^{-1}\sum_{\rho}w^{-\rho/k}\Gamma\left(\frac{\rho}{k}\right) + \mathcal{O}_{k}(\log (2qN)\log 2q),
        \] 
        where $\rho$ runs over the non-trivial zeros of $L(s,\chi)$. 
        In particular, for $\chi = \chi_{0}$ and $q\ge 2$,
        \begin{equation}
        \label{eq:chi0}
        \Psi_{k}(z, \chi_{0})=\sum_{n:(n,q)=1} \Lambda(n)z^{n^k} = \Psi_{k}(z) + \mathcal{O}(\log 2q\log N).
        \end{equation}
        \end{lemma}
        \begin{proof}
        We have to distinguish two possible scenarios:
                \begin{enumerate}
        \item $\chi=\chi_{0}$: by hypothesis, we get that
        \begin{align*}
         |\Psi_{k}(z,\chi_{0})-\Psi_{k}(z)|=\left|\sum_{n}(\chi_{0}(n)-1)\Lambda(n)z^{n^k} \right| \ll_{k} \log N\log 2q.   
        \end{align*}
        Applying the inverse Mellin transform and the residues theorem, we can express $\Psi_{k}(z,\chi_{0})$ as 
        \begin{align*}
        \Psi_{k}(z,\chi_{0})&=k^{-1}\Gamma\left(1+\frac{1}{k}\right)^{1/k}w^{-1/k} - k^{-1}\sum_{\rho}w^{-\rho/k}\,\Gamma\left(\frac{\rho}{k}\right) \\ &\qquad+ \mathcal{O}_{k}(\log 2qN\log 2q).
        \end{align*}
        \item $\chi \ne \chi_{0}$: we adapt Lemma 2 of \cite{lan12}, with $e^{-n^k /N}e(n^k\alpha)$, in the definition of $W(\chi, \eta,z)$, for $k>1$. Referring to Chapter $14$ of Davenport \cite{dav00}, we can estimate $L(s,\chi)$, for $s\in\mathbb{C}$, with $L(s,\chi_{1})$, if $\chi$ is not a primitive character and $\chi_{1}$ induces $\chi$. In fact, for every $\sigma>0$, we have
        \[
        \left| \frac{L'}{L}(s, \chi) - \frac{L'}{L}(s,\chi_{1})\right|\le \sum_{p|q} p^{-\sigma}\frac{\log p}{1-p^{-\sigma}}\le \sum_{p|q}\log p\le \log q.
        \]
     Similarly, we have that 
        \[
        |\Psi_{k}(z,\chi)-\Psi_{k}(z,\chi_{1})|\le \sum_{\substack{n\ge 1: \\ (n,q)>1 \\ (n, q_{1})=1}} \Lambda(n)e^{-n^k /N}\ll \log N\log q.
        \]
        We will use these inequalities for the third summand of the following approximation of $\Psi_{k}(z,\chi)$, (see Chapter 19 of \cite{dav00}):
        \begin{align*}
        \Psi_{k}(z,\chi)&=-k^{-1}\sum_{\rho}w^{-\rho/k}\Gamma(\rho/k) + C(\chi) \\ 
        &\qquad- \frac{1}{2\pi ik}\int_{(-k\sigma)} \frac{L'}{L}(u, \chi)\Gamma\left(\frac{u}{k}\right)w^{-u/k}\,du,
        \end{align*}
        \noindent where $C(\chi)$ only depends on the character $\chi$ and  $\sigma=3/4k$.
        By equations $(1)$ and $(4)$ of Chapter 16 of \cite{dav00}, we know that, for $-1\le(\mathfrak{R}(u)=-k\sigma)\le 2$ and for a primitive character $\chi_{1}$, 
        \begin{align*}
        \frac{L'}{L}\left(u, \chi_{1}\right) &= \sum_{\rho:\,|t-\gamma| < 1} \frac{1}{u-\rho}\, +\, \mathcal{O}(\log(q(|t|+2)) \ll \log(q(|t|+2)),
        \end{align*}
         because each summand is $\ll 1$ on the integration line $u=-k\sigma +i t$ and the number of non-trivial zeros $\rho$, in the interval $|t-\gamma|<1$, is $\mathcal{O}(\log q(|t|+2))$, (see (3) of Chapter 16 of \cite{dav00}). Thus,
        \[
        \left|\frac{L'}{L}\left(-\frac{3}{4}+it, \chi_{1}\right)\right|\ll \log(q(|t|+2)).
        \]
        Moreover, as in Lemma 2 of \cite{lan12}, it holds that
        \[
        |w^{-u/k}| = |w|^{\sigma}\exp{(tk^{-1}\arg{(w)})}, \quad\text{for $|\arg{(w)}|\le \frac{\pi}{2}$},
        \]
        and, by the Stirling formula, $\Gamma(u/k)\ll (|t|/k)^{-\sigma-1/2}\exp{(-|t|\pi/2k)}$. Hence,
        
        \begin{align*}
        \frac{1}{2\pi k}&\int_{(-k\sigma)}\left|-\frac{L'}{L}\left(u, \chi\right)\Gamma\left(\frac{u}{k}\right)w^{-u/k}\right|\,du \\ \\ &\le \frac{1}{2\pi k}\int_{(-k\sigma)}\left|-\frac{L'}{L}\left(u, \chi\right)+\frac{L'}{L}\left(u, \chi_{1}\right)\right|\cdot\left|\Gamma\left(\frac{u}{k}\right)w^{-u/k}\right|\,du \\ \\ &\qquad+\frac{1}{2\pi k}\int_{(-k\sigma)}\left|\frac{L'}{L}\left(u, \chi_{1}\right)\right|\cdot\left|\Gamma\left(\frac{u}{k}\right)w^{-u/k}\right|\,du \\ \\ &\ll\frac{\log q}{2\pi k}\int_{(-k\sigma)}\left|\Gamma\left(\frac{u}{k}\right)w^{-u/k}\right|\,du  \\ \\ &\qquad+\frac{1}{2\pi k}\int_{(-k\sigma)}\left|\frac{L'}{L}(u,\chi_{1})\Gamma\left(\frac{u}{k}\right)w^{-u/k}\right|\,du \\ \\ 
        &\ll_{k} \log q \cdot|w|^{\sigma} \int_{0}^{+\infty}t^{-\sigma-1/2}\exp\left(\frac{t}{k}\left(\arg(w)-\frac{\pi}{2}\right)\right)\,dt \\ \\ 
        &\qquad+ |w|^{\sigma}\int_{0}^{1}\log(q(t+2))\,dt  \\ \\ &\qquad+ |w|^{\sigma}\int_{1}^{+\infty} \log(q(t+2))\cdot t^{-\sigma-1/2}\exp\left(\frac{t}{k}\left(\arg(w)-\frac{\pi}{2}\right)\right)\,dt \\ \\ 
        &\ll_{k} |w|^{\sigma} (\log q + 1+\log q +1) = 2|w|^{\sigma}(1+\log q),
        \end{align*}
        by the fact that $w=N^{-1}-2\pi i\alpha\ll 1$. \newline For $C(\chi)$, we can argue as in Lemma 2 of \cite{lan12}, which is based on Chapter 19 of \cite{dav00}, and conclude that 
        \[
        C(\chi)=
        \begin{cases}
         -(L'/L)(0,\chi) & \text{if $\chi$ is odd} \\
         \log w -b(\chi)-\Gamma'(1) &\text{if $\chi$ is even},
        \end{cases}
        \]
        where $\log w\ll 1$ and
        \[
        b(\chi)=-\sum_{\rho} \left(\frac{1}{\rho}-\frac{1}{2-\rho}\right) \,+\,\mathcal{O}(1) = -\sum_{|\gamma|<1}\frac{1}{\rho} + \mathcal{O}(\log 2q) \ll \log^2 2q. 
        \] Finally, we have the estimate, as
        \[
        \Psi_{k}(z,\chi)=-k^{-1}\sum_{\rho}w^{-\rho/k}\Gamma(\rho/k) + E(q,N),
        \]
        where
        \[
        E(q,N)=\begin{cases}
         1+\log^2 2q &\text{if $\chi$ is a primitive character} \\
         (\log N)(\log 2q) + \log^2  2q & \text{if $\chi$ is not primitive. }
        \end{cases}
        \]
        \end{enumerate}
        \end{proof}
       Now, mimicking the strategy of \cite{ike25}, by Lemmas \ref{th3}, \ref{lem1} and formula \eqref{eq:Gqk}, we decompose $G_{q,k}(N)$ as
\begin{align}
\notag
  G_{q,k}(N)
  &=
  \int_{0}^{1} \biggl(\frac{1}{\phi(q)} \sum_{\chi \bmod q} \chi(-1)\Psi_{k}(z,\chi^k)\Psi_{k}(z,\overline{\chi}^k) + \mathcal{O}((\log N\log 2q)^2)\biggr) \cdot \\
\notag
  &\qquad\cdot
  I_{N}(z^{-1})\,d\alpha \\
  &=:
  \mathcal{I}_{1} +\mathcal{I}_{2} + \mathcal{O}(\mathcal{I}_{3}), 
\label{eq:stim}
\end{align}
where
\begin{equation}
\label{eq:mathcalI1}
\mathcal{I}_{1} =\frac{1}{\phi(q)}\sum_{\chi^k =\chi_{0}}\chi(-1)\int_{0}^{1} |\Psi_{k}(z,\chi_{0})|^2I_{N}(z^{-1})\,d\alpha, 
\end{equation}
\begin{equation}
\label{eq:mathcalI2}
\mathcal{I}_{2} =\frac{1}{\phi(q)} \sum_{\substack{\chi \bmod q \\ \chi^k\ne\chi_{0}}} \chi(-1)\int_{0}^{1} \Psi_{k}(z,\chi^k)\Psi_{k}(z,\overline{\chi}^k) I_{N}(z^{-1})\,d\alpha
\end{equation}
and
\begin{equation}
\label{eq:mathcalI3}
\mathcal{I}_{3} =(\log N\log 2q)^2\int_{0}^{1} |I_{N}(z^{-1})|\,d\alpha.
\end{equation}

\subsection*{Sketch of the proof}
\begin{itemize}
\item By computations in (2.4)-(2.7) of \cite{ike25}, we know that 
\begin{equation}
\label{eq:I_N}
|I_{N}(z^{-1})|\le e \min\{N,|\alpha|^{-1}\} \implies \int_{0}^{1}|I_{N}(z^{-1})|\,d\alpha=\mathcal{O}(\log N), 
\end{equation}
thus, from \eqref{eq:mathcalI3}, we conclude that $\mathcal{I}_{3}=\mathcal{O}(\log^2 2q \cdot\log^3 N)$.
\item Recalling the asymptotic relation \eqref{eq:chi0} of Lemma \ref{lem1} 
\[
\Psi_{k}(z,\chi_{0})=\Psi_{k}(z)+\mathcal{O}(\log 2q\log N),
\]
we get that $\mathcal{I}_{1}$ contains all the main terms of $G_{q,k}(N)$, as
\begin{align*}
\mathcal{I}_{1}&=\frac{1}{\phi(q)}\sum_{\chi^k =\chi_{0}}\chi(-1) \int_{0}^{1}(\Psi_{k}(z)+\mathcal{O}(\log q\log N))^2 I_{N}(z^{-1})\,d\alpha  \\
&=\left(\frac{1}{\phi(q)}\sum_{\chi^k =\chi_{0}}\chi(-1)\right) (G_{1,k}(N)+\mathcal{O}(N^{1/k}\log^2 N\log q \cdot \phi(q)^{-1})) \\ 
&=:\frac{\Sigma_{k}(q)}{\phi(q)}\cdot G_{1,k}(N)+\mathcal{O}(N^{1/k}\log^2 N\log q \cdot \phi(q)^{-1}),
\end{align*}
basing on (2.8)-(2.11) of \cite{ike25} and on the definition of $\Sigma_{k}(q)$, (see $\S$ \ref{sec:sigma} for its evaluation).

\item Arguing as in (2.12) of \cite{ike25}, by \eqref{eq:I_N}, we reduce to the estimate of
\begin{align*}
|\mathcal{I}_{2}|&\ll \max_{\chi^k\ne\chi_{0}}\sum_{0\le m<\log_{2}N}\frac{N}{2^m}\int_{0}^{2^{m+1}/N}|\Psi_{k}(z,\chi^k)|^2\,d\alpha \\ &=:\max_{\chi^k\ne\chi_{0}}\sum_{0\le m<\log_{2}N}\frac{N}{2^m}\cdot I_{m}.
\end{align*}
Denoting by $h=N/2^{m+2}$, we apply Gallagher's Lemma (see Lemma 1.9 of Montgomery \cite{mon71} and \S \ref{subsec3.3}) to $I_{m}$, i.e.
\begin{align*}
I_{m}&\ll \frac{1}{h^2}\int_{0}^{h}\left|\sum_{n\le x}\chi(n)^k \Lambda(n)e^{-n^k/N}\right|^2\,dx \\ 
&\qquad+\frac{1}{h^2}\int_{0}^{+\infty}\left|\sum_{x<n\le x+h}\chi(n)^k \Lambda(n)e^{-n^k/N}\right|^2\,dx \\
&=:\frac{1}{h^2}(I_{1}(N,h)+I_{2}(N,h)).
\end{align*}
Thanks to partial summation, Lemma 2.4 of \cite{ike25} and the Cauchy-Schwarz inequality, we get that
\begin{align*}
\mathcal{I}_{2}\ll N^{1/k}\log^3 N\log^2 q < N^{1/k\,+1},
\end{align*}
(for the details, see \S\ref{subsec3.3}).
\end{itemize}

\section{Proof of Theorem \ref{th2}}\label{sec3}

\subsection
{\textbf{Evaluation of \texorpdfstring{$\mathcal{I}_{3}$}{I3}}}
As in $(2.4)-(2.7)$ of \cite{ike25}, by \eqref{eq:I_N}, $\mathcal{I}_{3}$ is estimated as we have already seen in the sketch of the proof.
\subsection{Evaluation of \texorpdfstring{$\mathcal{I}_{1}$}{I1}}
By Lemma \ref{th3}, formula \eqref{eq:chi0} of Lemma \ref{lem1} and \eqref{eq:g1k}, we can rewrite \eqref{eq:mathcalI1} as
\begin{equation}
\label{sec3:1}
\begin{split}
\mathcal{I}_{1}&=\frac{\Sigma_{k}(q)}{\phi(q)}\int_{0}^{1} (\Psi_{k}(z)+\mathcal{O}(\log N\log q))^2\,I_{N}(z^{-1})\,d\alpha \\ &
= \frac{\Sigma_{k}(q)}{\phi(q)}\cdot G_{1,k}(N) +\mathcal{O}\left(|\Sigma_{k}(q)|\frac{\log N\log q}{\phi(q)} \int_{0}^{1}|\Psi_{k}(z)I_{N}(z^{-1})|\,d\alpha\right) \\  
&\qquad+\mathcal{O}\left(|\Sigma_{k}(q)|\frac{\log^3 N\log^2 q}{\phi(q)}\right)\,.
\end{split}
\end{equation}
Then, by partial summation, we obtain that
\begin{align*}
|\Psi_{k}(z)|&\le \sum_{n}\Lambda(n)e^{-n^k/N} \ll \frac{k}{N}\int_{1}^{+\infty} \psi(t)\cdot t^{k-1}e^{-t^k/N}\,dt \\ &\sim \frac{k}{N}\int_{1}^{+\infty}t^{k}e^{-t^k/N}\,dt = N^{1/k}\int_{1/N}^{+\infty} u^{1/k}e^{-u}\,du \ll N^{1/k},
\end{align*}
by the change of variables $u=t^k/N$. Thus, by \eqref{eq:I_N} again, the second summand in \eqref{sec3:1} is 
\[
\ll \frac{N^{1/k}\log N\log q}{\phi(q)}\int_{0}^{1}|I_{N}(z^{-1})|\,d\alpha\ll N^{1/k}\frac{\log^2 N\log q}{\phi(q)},
\]
which let us express $\mathcal{I}_{1}$ as
\begin{align*}
&\mathcal{I}_{1}=\frac{\Sigma_{k}(q)}{\phi(q)}\cdot G_{1,k}(N) +\mathcal{O}\left(N^{1/k}\frac{\log^2 N\log q}{\phi(q)}\right),
\end{align*}
and, by Lemma \ref{app}, the main term of $\mathcal{I}_{1}$ vanishes according to the value of $\Sigma_{k}(q)$. \newline
Therefore, we can rewrite \eqref{eq:stim} as
\[
G_{q,k}(N)=\frac{\Sigma_{k}(q)}{\phi(q)}\cdot G_{1,k}(N) +\mathcal{O}\left(N^{1/k}\frac{\log^2 N\log q}{\phi(q)}\right) +  \mathcal{I}_{2}+\mathcal{O}(\log^3 N\log^2 q).
\]
\subsection{Evaluation of \texorpdfstring{$\mathcal{I}_{2}$}{I2}}
\label{subsec3.3}
By \eqref{eq:mathcalI2}, we have
\begin{align*}
|\mathcal{I}_{2}| &\le \frac{1}{\phi(q)}\int_{0}^{1} \sum_{\chi^k\ne\chi_{0}} |\Psi_{k}(z,\chi^k)|^2 |I_{N}(z^{-1})|\,d\alpha \\ &\le \max_{\chi^k\ne\chi_{0}}\int_{0}^{1} |\Psi_{k}(z,\chi^k)|^2 |I_{N}(z^{-1})|\,d\alpha.
\end{align*}
Recalling \eqref{eq:I_N}, as in (2.12) of \cite{ike25}, we get that
\begin{equation}
\label{eq:I2est}
\begin{split}
\mathcal{I}_{2} &\ll \max_{\chi^k\ne\chi_{0}}\int_{0}^{1/2} |\Psi_{k}(z,\chi^k)|^2\cdot \min\{N, |\alpha|^{-1}\}\,d\alpha \\ 
&\le\max_{\chi^k\ne\chi_{0}}\,N \int_{0}^{1/N}|\Psi_{k}(z,\chi^k)|^2\,d\alpha \\  
&\qquad+\max_{\chi^k\ne\chi_{0}}\,N\sum_{0\le m<\log_{2}N} \frac{1}{2^m}\int_{2^m/N}^{2^{m+1}/N} |\Psi_{k}(z,\chi^k)|^2\,d\alpha
\\ 
&\le N\sum_{0\le m< \log_{2}N} \frac{1}{2^m}\cdot \max_{\chi^k\ne\chi_{0}}\int_{0}^{2^{m+1}/N} |\Psi_{k}(z,\chi^k)|^2\,d\alpha.
\end{split}
\end{equation}
We apply Gallagher's Lemma to this integral, for a fixed $h>0$ that we will choose later, and we have
\begin{align*}
\int_{0}^{1/2h} |\Psi_{k}(z,\chi^k)|^2\,d\alpha &= \int_{0}^{1/2h} \left|\sum_{n}\chi(n)^{k}\Lambda(n)r^{n^k}e(n^k \alpha) \right|^2\,d\alpha \\ 
&\ll \frac{1}{h^2} \int_{-\infty}^{+\infty}\left|\sum_{x<n\le x+h}\chi(n)^{k}\Lambda(n)e^{-n^k/N}\right|^2\,dx \\
&=\frac{1}{h^2} \int_{0}^{h}\left|\sum_{n\le x}\chi(n)^{k}\Lambda(n)e^{-n^k/N}\right|^2\,dx \\ &\qquad+\frac{1}{h^2} \int_{0}^{+\infty}\left|\sum_{x<n\le x+h}\chi(n)^{k}\Lambda(n)e^{-n^k/N}\right|^2\,dx \\ 
&=:\frac{1}{h^2}(I_{1}(N,h)+I_{2}(N,h)).
\end{align*}
Thus, in order to conclude, we have to estimate $I_{1}(N,h)$ and $I_{2}(N,h)$.
Before that, we recall Lemma 2.4 of \cite{ike25}, which keeps holding for $\chi^k$, instead of $\chi$.
We introduce the function
\[
\psi_{k}(x,\chi):=\sum_{n\le x}\chi^k(n) \Lambda(n)
\]
and the two integrals
\[
J_{1}(X):=\int_{0}^{X}|\psi_{k}(x,\chi)|^2\,dx \quad\text{and}\quad J_{2}(X,h):=\int_{0}^{X}|\psi_{k}(x+h,\chi)-\psi_{k}(x,\chi)|^2\,dx.
\]
\begin{lemma}{(New version of Lemma  2.4 of \cite{ike25})}
\label{lem2}
Assuming GRH and $X\ge 1$, for any Dirichlet character $\chi$ s.t. $\chi^k\ne\chi_{0} \bmod q$, we have
\begin{equation*}
J_{1}(X)\ll X^2 \log^2 (2q)
\end{equation*}
and, for $0\le h\le X$,
\begin{equation*}
J_{2}(X,h)\ll (h+1)X \log^2 \left(\frac{3qX}{h+1}\right).
\end{equation*}
\end{lemma}

\begin{proof}
For $\chi^k$ primitive character, it is straightforward from Lemma $2$ of Goldston and Vaughan \cite{gol96} and Chapter 19 of Davenport \cite{dav00}, since the $k$-th power does not affect the argument based on Perron's formula
\[
\psi_{k}(z,\chi)=-\frac{1}{2\pi i}\int_{(c)}\frac{L'}{L}(s,\chi^k)\cdot\frac{x^s}{s}\,ds.
\]
Otherwise, we consider the primitive character $\chi_{1}$ that induces $\chi^k$, (as in Lemma \ref{lem1}) and we argue like in the proof of Lemma 2.4 of \cite{ike25}.
\end{proof}

\begin{itemize}
\item{\textbf{Evaluation of $I_{1}(N,h)$}}: \newline By partial summation, we have
\[
\sum_{n\le x} \chi^k (n)\Lambda(n)e^{-n^k/N} = \psi_{k}(x,\chi)e^{-x^k/N}+\frac{k}{N}\int_{0}^{x} \psi_{k}(u,\chi)u^{k-1}e^{-u^k/N}\,du.
\]
Applying the inequality
\begin{equation}
\label{sec3:4}    
|a+b|^2 \le 2|a|^2 + 2|b|^2\, \qquad \forall\,a,b\in \mathbb{C},
\end{equation} 
we get
\begin{align*}
I_{1}(N,h) &\le 2\int_{0}^{h} |\psi_{k}(x,\chi)|^2e^{-2x^k/N}\,dx \\ \\ &\quad+\frac{2k^2}{N^2}\int_{0}^{h} \left|\int_{0}^{x}\psi_{k} (u,\chi)u^{k-1}e^{-u^k/N}\,du \right|^2 \,dx \\ \\
&=: 2 [(i)+k^2 (ii)].
\end{align*}
For $(i)$, 
like in \cite{ike25}, we notice that 
\[
(i):=\int_{0}^{h}|\psi_{k}(x,\chi)|^2\,e^{-2x^k/N}\,dx \ll \int_{0}^{h}|\psi_{k}(x,\chi)|^2\,dx \ll J_{1}(h)\ll h^2 \log^2 (2q),
\]
by definition and property of $J_{1}$ in Lemma \ref{lem2}.
For $(ii)$, denoting by
\[
(ii)=:\frac{1}{N^2} \int_{0}^{h} S_{0,x}^2\,dx, 
\]
rewriting $e^{-u^k /N}=e^{-u^k /2N}e^{-u^k /2N}$ and by the Cauchy-Schwarz inequality, we have 
{\small
\begin{equation}
\label{eq:Sox}
S_{0,x}^2\le \left(\int_{0}^{x} b^2 (u)e^{-u^k/N}\,du \right)\left(\int_{0}^{x} g^2(u)e^{-u^k/N}\,du\right),
\end{equation}
}
\noindent for $b(u)=|\psi_{k}(u,\chi)|$ and $g(u)=u^{k-1}$.
Now, as in \cite{ike25}, we use the fact that, for any function $0\le f(x) \ll |x|^m$, (for some $m\in\mathbb{N}$), we have
\begin{equation}
\label{sec3:5}
\begin{split}
\int_{0}^{\infty} f(x)e^{-2x^k/N}\,dx &\le \int_{0}^{N^{1/k}} f(x)e^{-2x^k/N}\,dx \\ 
&\qquad+ \sum_{j=1}^{\infty} \int_{jN^{1/k}}^{(j+1)N^{1/k}} f(x)e^{-2x^k/N}\,dx \\ 
&\le \sum_{j=1}^{\infty} \frac{1}{2^{j-1}}\int_{0}^{jN^{1/k}} f(x)\,dx.
\end{split}
\end{equation}
Thus, if we take $f(u)=|\psi_{k}(u,\chi)|^2$ in \eqref{sec3:5}, we get that
\begin{equation}
\label{eq:f(i)}
\begin{split}
\int_{0}^{x}f (u)e^{-2u^k /N}\,du&\le \int_{0}^{\infty} |\psi_{k}(u,\chi)|^2 e^{-2u^k /N}\,dx \le \sum_{j=1}^{\infty} \frac{J_{1}(jN^{1/k})}{2^{j-1}} \\ &\ll N^{2/k}\log^2 (2q)\sum_{j=1}^{\infty} \frac{j^2}{2^j} \ll \log^2 (2q)N^{2/k},
\end{split}
\end{equation}
where $J_{1}$ is defined in Lemma \ref{lem2}. 
With the change of variables $t=u^{k}/N$, we can bound the other integral in \eqref{eq:Sox}, in terms of Gamma function, as
\begin{align*}
\int_{0}^{x}g^2 (u)e^{-u^k /N}\,du &=\int_{0}^{x} u^{2(k-1)}e^{-u^k/N}\,du\asymp_{k} N^{2-1/k} \int_{0}^{x^k /N} t^{1-1/k}e^{-t}\,dt \\
&\ll_{k} N^{2-1/k}\, \Gamma(2-1/k)\asymp_{k}N^{2-1/k}.
\end{align*}
Therefore, we get that
\[
(ii) \ll \frac{\log^2 (2q)}{N^2}\int_{0}^{h}N^{2/k}\cdot N^{2-1/k}\,dx = hN^{1/k}\log^2 (2q) ,
\]
and thus we conclude that
\begin{equation}
\label{eq:I1}
I_{1}(N,h)\ll (\log^2 (2q))\cdot(h^2+hN^{1/k}).
\end{equation}
\item {\textbf{Evaluation of $I_{2}(N,h)$:}} \newline
Again by partial summation and adding and subtracting $\psi_{k}(x+h,\chi)e^{-x^k /N}$,
\begin{align*}
\sum_{x<n\le x+h}\chi^k(n)\Lambda(n)e^{-n^k/N} &= 
(\psi_{k}(x+h,\chi) - \psi_{k}(x,\chi))e^{-x^k/N} \\ &\qquad+ \psi_{k}(x+h,\chi)(e^{-(x+h)^k/N} - e^{-x^k/N}) \\ &\qquad+ \frac{k}{N}\int_{x}^{x+h} \psi_{k}(u,\chi)u^{k-1}e^{-u^k/N}\,du.
\end{align*}
We further observe that
\begin{equation}
\label{eq:dif-exp}
|e^{-(x+h)^k /N}-e^{-x^k /N}|=\left|-\int_{x^k /N}^{(x+h)^k /N} e^{-u}\,du\right| \ll \frac{h}{N}x^{k-1}e^{-x^k /N},
\end{equation}
as in \cite{ike25}.
Therefore, using \eqref{sec3:4}, we can bound $I_{2}(N,h)$ as
\begin{equation}
\label{eq:I2h}
\begin{split}
I_{2}(N,h) &\ll_{k} \int_{0}^{+\infty} |\psi_{k}(x+h,\chi) - \psi_{k}(x,\chi)|^2e^{-2x^k/N}\,dx \\ 
&\qquad+\frac{h^2}{N^2}\int_{0}^{+\infty} |\psi_{k}(x+h,\chi)|^2 x^{2(k-1)}e^{-2x^k/N}\,dx \\  
&\qquad+ \frac{1}{N^2} \int_{0}^{+\infty}
\left(\int_{x}^{x+h} |\psi_{k}(u,\chi)|u^{k-1}e^{-u^k/N}\,du \right)^2\,dx,
\end{split}
\end{equation}
where we have replaced \eqref{eq:dif-exp} with our previous estimate. 
By definition of $J_{2}$ in Lemma \ref{lem2}, we find that
\begin{align}
\notag
\int_{0}^{+\infty} |\psi_{k}(x+h,\chi)-\psi_{k}(x,\chi)|^2 &e^{-2x^k /N}\,dx \ll \sum_{j=1}^{\infty}\frac{J_{2}(jN^{1/k},h)} {2^{j-1}} \\ \label{eq:sumI21}
&\ll (h+1)N^{1/k}\sum_{j\ge 1}j\log^2 \left(\frac{3q jN^{1/k}}{h+1} \right),
\end{align}
which gives us the bound for the first integral in \eqref{eq:I2h}.
For the second term in \eqref{eq:I2h}, applying \eqref{sec3:5} with $f(x)=|x^{k-1}\cdot\psi_{k}(x+h,\chi)|^2$, definition of $J_{1}$ in Lemma \ref{lem2} and by partial summation, we get that
\begin{equation}
\label{eq:I2hbis}
\begin{split}
\int_{0}^{+\infty}&|\psi_{k}(x+h,\chi)|^2 |e^{-(x+h)^k/N}-e^{-x^k /N}|^2\,dx \\ &\ll \frac{h^2}{N^2}\int_{0}^{+\infty}|\psi_{k}(x+h,\chi)|^2 x^{2(k-1)}e^{-2x^k /N}\,dx \\
&=\frac{h^2}{N^2}\int_{0}^{+\infty} f(x)e^{-2x^k/N}\,dx \\ &\ll \frac{h^2}{N^2}\sum_{j\ge 1}\frac{1}{2^{j-1}}\int_{0}^{jN^{1/k}} f(x)\,dx \\ &= \frac{h^2}{N^2}\sum_{j\ge 1} \frac{1}{2^{j-1}}\left(J_{1}(x)x^{2(k-1)}\bigg|_{0}^{jN^{1/k}}-\int_{0}^{jN^{1/k}}2(k-1)J_{1}(x)x^{2k-3}\,dx\right) \\ &\ll \frac{h^2}{N^2}\sum_{j\ge 1} \frac{1}{2^{j-1}}J_{1}(jN^{1/k})j^{2(k-1)}N^{2-2/k}\ll h^2\log^2 (2q)\sum_{j\ge 1} \frac{j^{2k}}{2^{j-1}}.
\end{split}
\end{equation}
Similarly, for the third summand in \eqref{eq:I2h}, 
denoting by 
\[
S_{x, \,x+h}:=\int_{x}^{x+h} |\psi_{k}(u,\chi)|u^{k-1}e^{-u^k/N}\,du,
\]
by the Cauchy-Schwarz inequality, again Lemma \ref{lem2} and changing the order of integration, we have 
\begin{equation}
\label{eq:Sxh}
\begin{split}
\frac{1}{N^2}\int_{0}^{+\infty}S_{x,x+h}^2\,dx &\ll \frac{h}{N^2}\int_{0}^{+\infty}\int_{x}^{x+h}u^{2(k-1)}|\psi_{k}(u,\chi)|^2 e^{-2u^k /N}\,du\,dx  \\&=\frac{h^2}{N^2} \int_{0}^{+\infty} x^{2(k-1)}|\psi_{k}(x,\chi)|^2 e^{-2x^k /N}\,dx  \\  
&\ll \frac{h^2}{N^2}N^2 \log^2 (2q)\sum_{j\ge 1}\frac{j^{2k}}{2^{j-1}}.
\end{split}
\end{equation}
By \eqref{sec3:5} and collecting \eqref{eq:sumI21}, \eqref{eq:I2hbis} and \eqref{eq:Sxh}, we get the following estimate for \eqref{eq:I2h}
\begin{equation}
\label{eq: stimafinI2}
\begin{split}
I_{2}(N,h) &\ll \sum_{j=1}^{\infty}\frac{1}{2^j}J_{2}(jN^{1/k},h) +2h^2 \log^2 (2q)\sum_{j\ge 1} \frac{j^{2k}}{2^j} \\ &\ll \sum_{j=1}^{\infty}\frac{1}{2^j}(h+1)jN^{1/k}\log^2 \left(\frac{3qjN^{1/k}}{h+1}\right) + 2 h^2\log^2 (2q) \\ 
&\ll hN^{1/k}\log^2 N + h^2\log^2 (2q).
\end{split}
\end{equation}
\end{itemize}
By \eqref{eq:I1} and \eqref{eq: stimafinI2}, we have that
\begin{align*}
\frac{1}{h^2}({I}_{1}(N,h)+I_{2}(N,h)) &\ll \frac{1}{h^2}((h^2+hN^{1/k})\log^2 (2q) +hN^{1/k}\log^2 N)\\ &\ll\frac{N^{1/k}}{h}\log^2 N,
\end{align*}
where $q$-terms do not appear because $q$ is fixed and thus its contribution in \newline $\log^2 (qjN^{1/k}/(h+1))$, in \eqref{eq: stimafinI2}, can be neglected. 
Eventually, picking $h=N/2^{m+2}$, s.t. $0\le m <\log_{2}N$ and $1/4 < h \le N/4$, we can substitute our estimates in \eqref{eq:I2est}
\begin{align*}
\mathcal{I}_{2} &\ll N\sum_{0\le m<\log_{2}N}\frac{1}{2^m}\left(\max_{\chi^k\ne \chi_{0}}\int_{0}^{2^{m+1}/N} |\Psi_{k}(z,\chi^k)|^2\,d\alpha \right) \\ 
&\ll N\sum_{0\le m<\log_{2}N}\left(\frac{2^{m+2}}{2^m}N^{1/k\,-1}\log^2 N\right)\ll  N^{1/k}\log^3 N.
\end{align*}
As a result, for $k\ge 1$, it is always guaranteed that $\mathcal{I}_{2}=\mathcal{O}( N\,^{1/k}\log^3N)$.

\section{Evaluation of \texorpdfstring{$\Sigma_{k}(q)$}{Sigma(q)}}
\label{sec:sigma}
\begin{lemma}
\label{app}
Let $k\ge 2$ and $q\ge 1$ fixed natural numbers and define 
\[\Sigma_{k}(q)=\sum_{\chi^{k}=\chi_{0}} \chi(-1).
\]
Then, $\Sigma_{k}(q)$ is a multiplicative function. Moreover, denoting by 

\[
\delta:=\frac{\varphi(p^{\alpha})}{(k, \varphi(p^{\alpha}))},
\]
\newline 
for $p\ge 3$ and $\alpha\ge 1$, we have that
\[
\Sigma_{k}(p^{\alpha})=\begin{cases}
(k, \varphi(p^{\alpha})) & \text{if $\delta$ is even}  \\ 
0 & \text{if $\delta$ is odd and $(k, \varphi(p^{\alpha}))$ is even} \\ 
1 & \text{if $\delta$ is odd and $(k, \varphi(p^{\alpha}))$ is odd},
\end{cases}
\]
while, for $p=2$, $\alpha\ge 1$ and $\beta:=\max\{\gamma\in\mathbb{N}_{\ge 0}\,:\,2^{\gamma}\mid k\}$, 
\[
\Sigma_{k}(2^{\alpha})=\begin{cases}
1 & \text{if $\alpha=1$ or if $k$ is odd} \\
0 & \text{if $\alpha=2$ and $k$ is even} \\
2^{\beta}=(2^{\beta}, \varphi(2^{\alpha-1})) & \text{if $3\le\beta < \alpha-2$ and $k$ is even} \\ 
(k, \varphi(2^{\alpha})) & \text{if $\beta\ge\alpha-2$ and $k$ is even}. \\ 
\end{cases}
\]
\end{lemma}

\begin{remark}
\label{rem:int} The proof of Lemma \ref{app} relies on the cyclic group structure of $\mathbb{Z}_{p^{\alpha}}^*$, for $p$ odd prime number. Thus, we can find a generator $g\in\mathbb{Z}_{p^{\alpha}}^*$ and compute the values of all the $\varphi(p^\alpha)$-powers of $\chi(g)$.
\end{remark}

\begin{remark}
\label{rem1}
By the construction of $F_{q,k}(z)$ in Theorem \ref{th3}, (in particular, by the orthogonality property of characters in \eqref{eq:ort}), studying $\Sigma_{k}(q)$ is equivalent to counting the number of solutions of the equation
\[
\begin{cases}
m_{1}^k + m_{2}^k\equiv 0  &\mod q \\ 
(m_{1}m_{2}, q)=1 &
\end{cases}
\quad\iff\quad
\left(\frac{m_{2}}{m_{1}} \right)^k \equiv -1 \mod q.
\]
\end{remark}
\begin{remark}
\label{rem:sys}
Both of these strategies for the proof of Lemma \ref{app} will be first analysed with $q=p^{\alpha}$, $p$ prime and $\alpha\ge 1$; then, by the multiplicativity of $\Sigma_{k}(q)$, (or by the Chinese remainder theorem), for every $q$.
\end{remark}

\begin{proof}[Proof of Lemma \ref{app}]
We first show the multiplicativity of $\Sigma_{k}(q)$, i.e. that $\Sigma_{k}(1)=1$ and that, if $\chi_{1},\dots,\chi_{t}$ are characters mod \,$ (q_{1},\dots,q_{t})$ s.t. $(q_{i},q_{j})=1$, $\forall\,i\ne j$, $i,j\in\{1,\dots,t\}$, then
\[
\Sigma_{k}(q):=\Sigma_{k}(q_{1}\dots q_{t})=\Sigma_{k}(q_{1})\cdot\dots\cdot\Sigma_{k}(q_{t}),
\]
for $\chi$ character mod \,$q$. In fact, 
\begin{align*}
\Sigma_{k}(1)=\sum_{\chi^{k}=\chi_{0} \bmod 1}  \chi(-1)=\chi(-1)=1 \iff \chi=\chi_{0}.
\end{align*}
Then, we recall that, by the multiplicativity of the characters, we have
\[
\chi(n; q) \iff \chi_{1}(n; q_{1})\cdot\dots\cdot\chi_{t}(n;q_{t}),
\]
where $\chi_{j}(n;q_{j})= \chi_{j}(n) \bmod q_{j}$. \newline Therefore,
\begin{align*}
\chi^k (n;q)=\chi_{0}(n;q) &\iff (\chi_{1}(n;q_{1})\dots\chi_{t}(n;q_{t}))^k =\chi_{0}(n;q) \\ 
&\iff (\chi_{1}(n;q_{1})\dots\chi_{t}(n;q_{t}))^k=\chi_{0}(n;q_{1})\dots\chi_{0}(n;q_{t}).
\end{align*}
Now, choosing $\bar{n}, n\in\mathbb{Z}^{*}_{q_{1}\cdots q_{t}}$ s.t.
\[
\begin{cases}
&\bar{n}\equiv n \mod q_{i} \\
&\bar{n}\equiv 1 \mod q_{j}, \quad\forall\,i\ne j, j=1,\dots,i-1,i+1,\dots,t,
\end{cases}
\]
since $(q_{i}, q_{j})=1, \, \forall\,i\ne j$, we obtain that, for $j=1,\dots,t$, 
\[
(\chi_{1}(\bar{n};q_{1})\dots\chi_{t}(\bar{n};q_{t}))^k=\chi_{0}(\bar{n};q_{1})\dots\chi_{0}(\bar{n};q_{t}) \implies \chi_{j}(\bar{n};q_{j})^k =\chi_{0}(\bar{n};q_{j}), 
\]
while the other implication is trivial.
As a consequence, we get
\begin{align*}
\Sigma_{k}(q)=\sum_{\chi^k =\chi_{0}}\chi(-1)=\sum_{ \chi_{1}^k = \chi_{0}}\chi_{1}(-1)\dots\sum_{ \chi_{t}^k = \chi_{0}}\chi_{t}(-1) 
=\Sigma_{k}(q_{1})\dots\Sigma_{k}(q_{t}).
\end{align*}
As a result, it is enough to compute $\Sigma_{k}(q)$ for $q=p^{\alpha}$ in two cases:
\begin{itemize}
\item{$q=p^{\alpha},\, \alpha\ge 1,\, p\ge 3$};
\item{$q=2^{\alpha},\, \alpha\ge 1$}.
\end{itemize}
For $q=p^{\alpha}$, $p\ge 3$, let $g\in\mathbb{Z}_{q}^*$ be s.t. $\mathbb{Z}_{q}^* = \left \langle g \right \rangle $. By the construction of the characters mod $p^{\alpha}$, we know that, given $\chi \bmod\,p^{\alpha}$, there exists a $\varphi(p^{\alpha})$-th root of unity $\omega$ such that $\chi(g)=\omega$. By Euler's theorem, we also have
\[
g^{\varphi(q)}=1
\implies g^{\varphi(q)/2}=-1\implies \chi(g^{\varphi(q)/2})=\chi(-1).\]
Moreover, we know that 
\[
\chi^k =\chi_{0} \implies \chi^k (g)=\chi(g^k)=1,
\]
which means that $\chi(g)$ is a $k$-th root of unity. Thus, we get
\[
\begin{cases}
\chi(g^{\varphi(q)})=\chi(g)^{\varphi(q)}=1 & \\ \\
\chi(g^k)=\chi(g)^k =1 
\end{cases}\Longrightarrow\quad \exists\,j\in\{0,\dots,\varphi(q)-1\}\,:\, \chi(g)^k = 1 =(\omega^j)^k,
\]
In particular, we obtain that
\[
\omega^{jk}=1 \iff \varphi(q)\mid jk
\]
or, equivalently, that $j$ is a solution of the following equation:
\[
\begin{cases}
kx\equiv 0  \mod \varphi(q) \\ \\
x\in\mathbb{Z}_{\varphi(q)}
\end{cases} \quad \Longleftrightarrow \quad x\equiv 0 \mod \delta,
\]
where $\delta$ has been defined in the statement.
We define the set of all its $(k,\varphi(q))$ solutions as
\[
A_{k}(q):=\{0, \delta, 2\delta, \dots, ((k,\varphi(q))-1)\delta\,\}.
\]
Therefore, we can rewrite 
\[
\Sigma_{k}(q)=\sum_{j\in A_{k}(q)} \omega^{j\varphi(q)/2}.
\]
Furthermore, we notice that
\[
\omega^{\delta\varphi(q)/2}=
\begin{cases}
1 & \text{if $\delta$ is even} \\ 
-1 & \text{if $\delta$ is odd}.
\end{cases}
\]
As a result, we have that, if $\delta$ is even, then
\begin{align*}
\Sigma_{k}(q)&=1+\omega^{\delta\varphi(q)/2}+\dots+\omega^{\delta((k,\varphi(q))-1)/2} \\ 
&=1+\dots+(1)^{(k,\varphi(q))-1} = (k,\varphi(q)).
\end{align*}
Instead, if $\delta$ is odd, 

\[
\Sigma_{k}(q)=1-1+\dots+(-1)^{(k,\varphi(q))-1}=\begin{cases}
0 & \text{if $(k,\varphi(q))$ is even} \\ 
1 & \text{if $(k,\varphi(q))$ is odd}.
\end{cases}
\]
For $q=2^{\alpha}$, it is better to directly compute the first two cases:

\begin{enumerate}
\item \textbf{\emph{q=2:}} we only have one character, (by the fact that $\varphi(2)=1$), given by
\[
\chi(n)=\begin{cases}
1 & \text{if $(n,2)=1$ $\iff$ $n$ is odd} \\ 
0 & \text{if $(n,2)>1$ $\iff$ $n$ is even}.
\end{cases}
\]
Therefore, 
\[
\chi^k (n)=\begin{cases}
1 & \text{if $n$ is odd} \\
0 & \text{if $n$ is even},
\end{cases}
\]
which implies that 
\[
\Sigma_{k}(2)=\chi(-1)=1.
\]
\item \textbf{\emph{q=4:}} By Davenport's construction of characters,  (see \cite{dav00}, Chapters 1 and 4), we get $\varphi(4)=2$ characters mod 4
\[
\chi_{j}(n)=\omega^{\nu(n)} \quad:\quad \omega^{2}=1, \quad j=0,1,
\]
where $\nu$ generally denotes the index of $n$ with respect to a fixed primitive root of mod $p^{\alpha}$.
More precisely, we obtain $\chi_{0}$ - the principal character - and $\chi_{1}$ s.t.
\[
\chi_{1}(n)=
\begin{cases}
1 & \text{if $n\equiv 1 \mod 4$} \\
-1 & \text{if $n\equiv -1 \mod 4$}
\end{cases}
\]
and both of them are such that $\chi_{j}^k =\chi_{0}$.
As a result, 
\[
\Sigma_{k}(4)=\chi_{0}(-1)+\chi_{1}(-1)=1-1=0.
\]
\end{enumerate}
If $q=2^{\alpha}$, $\alpha\ge 3$, there exist $\varphi(2^{\alpha})=2^{\alpha-1}$ characters $\chi_{j} \mod 2^{\alpha}$, for $j=0,\dots,2^{\alpha-1}-1$, of the form
\[
\chi_{j}(n)=\omega^{\nu(n)}(\omega')^{\nu'(n)} \quad :\quad 
\begin{cases}
\omega^2=\omega^{\varphi(4)} =1 \\ 
(\omega')^{2^{\alpha-2}}=(\omega')^{\varphi(2^{\alpha-1})}=1.
\end{cases}
\]
Arguing as before, it holds that
\begin{align*}
&\chi_{j}^k (n)=\chi_{0}(n)=1 \iff \omega^{k\nu(n)}(\omega')^{k\nu'(n)}=1 \iff \begin{cases}
&2\mid k\nu(n) \\ \\ 
&2^{\alpha-2}\mid k\nu'(n)
\end{cases} \\ \\ &\iff \begin{cases}
&kx_{1}\equiv 0 \mod 2 \\ 
&kx_{2}\equiv 0 \mod 2^{\alpha-2}
\end{cases} \iff 
\begin{cases}
&x_{1}\equiv 0  \mod \left(\frac{2}{(k,2)} \right) \\ 
&x_{2}\equiv 0 \mod \left(\frac{2^{\alpha-2}}{(k, 2^{\alpha-2})} \right).
\end{cases}
\end{align*}
Thus, we should distinguish three possibilities, according to the value of 
\[
\beta:=\max\{\gamma\in\mathbb{N}_{\ge 0}\,:\, 2^{\gamma}\mid k\}.
\]
\begin{itemize}
\item If $\beta=0$, then $(k,2)=1$, i.e. $k$ is odd, which means that there exists exactly one solution of the system. In other words, there is a unique character $\chi_{j}$ mod $2^{\alpha}$ s.t. $\chi_{j}^k = \chi_{0}$ and, therefore, $\Sigma_{k}(2^{\alpha})=1$.
\item If $\beta\ge 2^{\alpha-2}$, then $k$ is even and the system is satisfied by every couple $(x_{1}, x_{2})\in\mathbb{Z}_{2^{\alpha-1}}\times \mathbb{Z}_{2^{\alpha-1}}$. We conclude that $\Sigma_{k}(2^{\alpha})=(k, \varphi(2^{\alpha}))$.
\item If $1\le \beta < \alpha-2$, $2^{\beta}\mid k$ but $2^{\beta+1}\nmid k$, thus 
\begin{align*}
\begin{cases}
&x_{1}\equiv 0 \mod \left(\frac{2}{(k,2)}=1\right) \quad \forall\, x_{1}\in\mathbb{Z} \\ 
&x_{2} \equiv 0 \mod \left(\frac{2^{\alpha-2}}{(k,2^{\alpha-2})}\right) \iff x_{2} \equiv 0 \mod \left(\frac{2^{\alpha-2}}{2^{\beta}}=2^{\alpha-\beta-2}=:N_{2}\right).
\end{cases}
\end{align*}
In this way, we get that $\nu=0$, while $\nu'\in A_{k}(2^{\alpha-2})$, for
\[
A_{k}(2^{\alpha-2})=\{0, N_{2}, \dots, N_{2}\cdot (2^{\beta}-1)\}.
\]
Moreover, we notice that $N_{2}$ is always even, thus
\[
(\omega')^{N_{2}\varphi(2^{\alpha-1})/2}=(\omega')^{N_{2}2^{\alpha-3}}=(\omega')^{2^{2(\alpha-2)-(\beta+1)}}=1
\]
and consequently
\begin{align*}
\Sigma_{k}(2^{\alpha})&=\sum_{\nu'} (\omega')^{\nu'\varphi(2^{\alpha-1})/2} = 1+(\omega')^{N_{2}2^{\alpha-3}}+\dots+(\omega')^{N_{2}2^{\alpha-3}(2^{\beta}-1)} \\ 
&=1+1+\dots+(1)^{2^{\beta}-1}=2^{\beta}.
\end{align*}
Equivalently, as $k=2^{\beta}h$, for $h$ odd, we can look at the equation
\[
z^k \equiv_{q} -1 \iff (z^{2^{\beta}})^h\equiv_{q} (-1)^h \iff z^{2^{\beta}} \equiv_{q} -1.
\]
But we have already seen that, for $k_{1}:=2^{\beta}$ even, there are $(k_{1}, \varphi(2^{\alpha}))$ \newline =\,$(2^{\beta}, 2^{\alpha-1})=2^{\beta}$ solutions. Thus, we have easily confirmed our result, using the second point of view of Remark \ref{rem:sys}.
\end{itemize}
\end{proof}

\bigskip

  AM: Dipartimento di Matematica e Informatica,
  Universit\`a degli Studi di Ferrara, Ferrara, Italy.
  email: \texttt{alessand.migliaccio@edu.unife.it}.
  AZ: [Corresponding Author]
  Dipartimento di Scienze Matematiche, Fisiche e Informatiche,
  Universit\`a di Parma, Parco Area delle Scienze, 53/a, 43124 Parma, Italy
  email: \texttt{alessandro.zaccagnini@unipr.it}

\end{document}